\documentclass{asl}

\newcommand{\K}{\lbar{K}}
\newcommand{\LOL}{\lbar{L}}

 \theoremstyle{plain}
 \newtheorem{thm}{Theorem}[section]
 \newtheorem{cor}[thm]{Corollary}
 \newtheorem{lem}[thm]{Lemma}
 \newtheorem{fact}[thm]{Fact}
 \newtheorem{prop}[thm]{Proposition}
 
\theoremstyle{definition}
 \newtheorem{defn}[thm]{Definition}
 
 \newtheorem{crit}[thm]{Criterion}
\theoremstyle{remark}
 \newtheorem{rem}[thm]{Remark}
 \newtheorem{ques}[thm]{Question}

 \numberwithin{equation}{section}
 \renewcommand{\theequation}{\thesection.\arabic{equation}}

\DeclareMathOperator{\MM}{\mathcal{M}}
\DeclareMathOperator{\LL}{\mathcal{L}}
\DeclareMathOperator{\OO}{\mathcal{O}}

\DeclareMathOperator{\CC}{\mathcal{C}}

\DeclareMathOperator{\rk}{rk}

 \DeclareMathOperator{\Th}{Th}

 \DeclareMathOperator{\lh}{lh}

 \DeclareMathOperator{\td}{tr\,deg}

 \DeclareMathOperator{\cha}{char}
 \DeclareMathOperator{\ac}{\overline{ac}}

\newcommand{\Z}{\mathbb{Z}}

\newcommand{\R}{\mathbb{R}}
\newcommand{\Q}{\mathbb{Q}}

\newcommand{\p}{$p$\nobreakdash}

 \newcommand{\abs}[1]{\left\vert#1\right\vert}
 \newcommand{\set}[1]{\left\{#1\right\}}
 \newcommand{\seq}[1]{\left<#1\right>}

\newcommand{\limplies}{\rightarrow}

\newcommand{\liff}{\leftrightarrow}
\newcommand{\ex}[1]{\exists #1 \;} 
\newcommand{\fa}[1]{\forall #1 \;} 

\newcommand{\proves}{\vdash}

\newcommand{\fun}{\longrightarrow}
\newcommand{\efun}{\longmapsto}
\newcommand{\sub}{\subseteq}

\newcommand{\mi}{\setminus}

\newcommand{\lbar}{\overline}

\newcommand{\tsub}{\mathbin{\mathchoice
{\buildrel .\lower.6ex\hbox{\vphantom{.}} \over {\smash-}}%
{\buildrel .\lower.6ex\hbox{\vphantom{.}} \over {\smash-}}%
{\buildrel .\lower.4ex\hbox{\vphantom{.}} \over {\smash-}}%
{\buildrel .\lower.3ex\hbox{\vphantom{.}} \over {\smash-}}}}

\newcommand{\N}{\mathbb{N}}

\title{On logical characterization of henselianity}

\author{Yimu Yin}
\revauthor{Yin, Yimu}

\address{Department of Philosophy\\
Carnegie Mellon University\\
Pittsburgh, PA 15213, USA}

\email{yimu.yin@gmail.com}

\begin{document}

\begin{abstract}
We give some sufficient conditions under which any valued field
that admits quantifier elimination in the Macintyre language is
henselian. Then, without extra assumptions, we prove that if a
valued field of characteristic $(0,0)$ has a $\Z$-group as its
value group and admits quantifier elimination in the main sort of
the Denef-Pas style language $\mathcal{L}_{RRP}$ then it is
henselian. In fact the proof of this suggests that a quite large
class of Denef-Pas style languages is natural with respect to
henselianity.
\end{abstract}

\maketitle


\section{Introduction}

One of the most important tools in model-theoretic algebra is
quantifier elimination (QE). Tarski's Theorem laid the foundation
for the subsequent work along this line:

\begin{thm}[Tarski]\label{tarski:theorem}
The theory RCF of real closed fields, as formulated in the
language $\mathcal{L}_{OR}$ of ordered rings, admits QE.
\end{thm}

Much later Macintyre proved a very important analog of this result
for \p-adic fields in~\cite{Ma76}:

\begin{thm}[Macintyre]\label{macintyre:theorem}
The theory of \p-adic fields, as formulated in the language
$\mathcal{L}_{Mac}$, admits QE.
\end{thm}

A crucial question for the algebraic structure of a field is of
course under what conditions polynomials have roots. Properties
that answer this question in real closed fields and \p-adic fields
are essential to the proofs of the above two theorems. They are of
course real-closedness and henselianity, respectively. One may
raise the question: Is QE equivalent to these properties after
all? For real closed fields there is a good answer:

\begin{thm}[Macintyre, McKenna, van den
Dries]\label{converse:qe:realclosed} Let $K$ be an ordered field
such that the theory of $K$ in $\mathcal{L}_{OR}$ admits QE. Then
$K$ is real closed.
\end{thm}

This result is established in~\cite{MMV83}, in which the authors
actually give a quite general technique that can be used to
establish other similar ``converse QE'' results for various kinds
of fields. In particular they have the following analogous result
for \p-fields:

\begin{thm}[Macintyre, McKenna, van den
Dries]\label{converse:qe:p-field} Let $K$ be a \p-field such that
the theory of $K$ in $\mathcal{L}_{Mac}$ admits QE. Then $K$ is
\p-adically closed.
\end{thm}

The definition of a \p-field $K$ is rather special: it is a
substructure of a \p-adically closed field $L$ (of \p-rank 1) with
respect to $\mathcal{L}_{Mac}$. The point is that, as $L$ is
henselian, each $n$th power predicate $P_n$ defines a clopen
subset of $K$ in the valuation topology of $K$, which is essential
to the proof of the theorem. This way to interpret each $P_n$ is
obviously unsatisfactory since an element in $P_n$ may not be an
$n$th power at all in $K$. Hence it is asked in~\cite{MMV83} to
extend the result to the class of valued fields where $P_n$ is
simply interpreted as the group of $n$th powers. In
Section~\ref{section:macintyre} we shall give some sufficient
conditions under which any such valued field that admits QE in
$\mathcal{L}_{Mac}$ is henselian. This addresses a question
in~\cite{cluckers:2002}. In fact this result holds for certain
finitely generated valued fields without QE; see
Section~\ref{section:henselianity:without:qe}.

There are variations and extensions of $\mathcal{L}_{Mac}$ in
which QE results for larger classes of valued fields have been
obtained, for example,~\cite{F:delon:1981,PrRo84}. There are yet
more languages which give rise to different techniques of QE in
valued fields and which cannot be subsumed under the Macintyre
style. The most notable among these is the Denef-Pas style, a
mature form of which is given in~\cite{Pa89}. In
Section~\ref{section:denef:pas} we shall show that any valued
field that admits QE in the main sort in the prototypical
Denef-Pas language $\mathcal{L}_{RRP}$, which is introduced
in~\cite{Pa89}, is henselian. In fact the proof of this suggests
that the result holds for a quite large class of Denef-Pas style
languages. This answers a question mentioned
in~\cite{cluckers:2002}.

Finally in Section~\ref{section:general:pers} a general
perspective on QE and converse QE results is described.

\section{Preliminaries}\label{section:prelim}

In this paper all valued fields are of characteristic 0 and all
valuation rings are proper subrings. We use $\OO$, $\OO_1$, etc.
and $\MM$, $\MM_1$, etc. to denote valuation rings and their
maximal ideals, respectively. Valuation maps are denoted by $v$,
$v_1$, etc. If $v$ is a valuation of $K$ then $vK$, $\K$ stand for
the corresponding value group and residue field, respectively.

The Macintyre language $\mathcal{L}_{Mac}$ for valued fields
contains the language of rings $\mathcal{L}_R$, $\{+, -, \cdot, 0,
1\}$, a unary predicate $\OO$ for valuation rings, and unary
predicates $P_n$ for all $n > 1$, which are usually interpreted as
the sets of nonzero $n$th powers.

\begin{defn}
Let $d$ be a fixed natural number. A \emph{\p-adically closed field
of \p-rank $d$} is a valued field such that
\begin{enumerate}
 \item the value group is a $\Z$-group with least positive element 1;
 \item the dimension of the $\mathbb{F}_p$-module $\OO/(p)$ is $d$, which is to say that
 the residue field is a finite extension of $\mathbb{F}_p$ of dimension $f$,
 $v(p) = e \cdot 1$ for some $e \in \N$, and $d = e \cdot f$;
 \item Hensel's Lemma holds.
\end{enumerate}
\end{defn}

Prestel and Roquette extended Theorem~\ref{macintyre:theorem} to
the class of \p-adically closed fields of finite \p-ranks,
providing that for each \p-rank $d$ one expands
$\mathcal{L}_{Mac}$ by adding $d$ new constants that serve as a
$\mathbb{F}_p$-basis of $\OO/(p)$; see~\cite[Theorem~5.6]{PrRo84}.

The proof of Thereom~\ref{converse:qe:p-field} relies on the
approximation technique devised in~\cite{MMV83}. In general this
technique consists of the following three steps. Let $(K, v)$ be a
valued field such that $\Th(K)$ admits QE (in the main sort) in
some language for valued fields, where $\Th(K)$ denotes the theory
of $K$ as a structure of the language in question. Let $\OO, \MM$
be its valuation ring and maximal ideal. For convenience,
throughout this paper, by valuation topology we mean the topology
on $K^{\times}$ (instead of $K$) that is induced by the valuation;
see Remark~\ref{why:tau*}.

\begin{itemize}

\item \textbf{Step 1}. Fix a syntactical notion of ``simple'' formulas.
This usually includes all the literals. Show that all
``simple'' formulas, except equations in the field, define
open sets in (the product of) the valuation topology. This is
where the rather special interpretation of $P_n$ in a \p-field
$K$ is needed in~\cite{MMV83}, which guarantees that $P_n$ is
a clopen subgroup of $K^{\times}$. Note that $P_n$ is not
closed in the valuation topology on $K$ as there is no open
neighborhood of $0$ that does not intersect with $P_n$. Also
note that, for each formula $\varphi(X)$, that it defines an
open set can be expressed by a first-order sentence:
\[
\fa{X} (\varphi(X) \limplies \ex{Y} (v(Y) > v(X) \wedge \fa{Z}
(v(Z) > v(Y) \limplies \varphi(X + Z)))).
\]

\item \textbf{Step 2}. Suppose that a monic polynomial $F(X,
\bar{a}) \in \OO[X]$ is a counterexample to a version of
Hensel's Lemma, where $\bar{a}$ are the (nonzero)
coefficients. For example, $F(s, \bar{a}) \in \MM$ but $F'(s,
\bar{a}) \notin \MM$ for some $s \in \OO$ and $F(X, \bar{a})$
has no root in $K$. By assumption, the formula that defines
the tuples of the coefficients of all such counterexamples for
a fixed degree is equivalent to a formula $\varphi$ that is
quantifier-free (in the main sort) and is in disjunctive
normal form. Through some algebraic manipulations it can be
shown that one of the disjuncts $\varphi_0$ of $\varphi$
defines a nonempty set $\varphi_0(K^{n})$ that is not
contained in any proper Zariski closed subset of $K^{n}$; that
is, $\varphi_0$ lacks equational conditions and hence, by
Step~1, defines a nonempty open set in $K^{n}$. Without loss
of generality $\bar{a} \in \varphi_0(K^{n})$. For details
see~\cite[Theorem~1,~4]{MMV83}.

\item \textbf{Step 3}. If $K$ is dense in its henselization
$K^h$ then the approximation can be carried out as follows:
Choose a root $r \in K^h$ of $F(X, \bar{a})$ and write
\[
F(X, \bar{a}) = (X + r)F^*(X, \bar{b}),
\]
where $\bar{b} \in K^h$ are the (nonzero) coefficients of
$F^*$. Let $U \sub \varphi_0(K^{n})$ be an open neighborhood
of $\bar{a}$, where $\varphi_0$ is as in Step~2. Now we can
choose $r', \bar{b}' \in K$ that are arbitrarily close to $r,
\bar{b}$ with respect to the valuation. Write
\[
F(X, \bar{a}') = (X + r')F^*(X, \bar{b}').
\]
So $\bar{a}' \in U$, which contradicts the choice of $U$.

However, in general $K$ is not dense in its henselization. The
solution to this in~\cite{MMV83} is to consider the field $A$
of algebraic numbers of $K$. By the assumptions there, in
particular that $K$ is a \p-field, $A$ cannot be henselian. On
the other hand, $A$ has $\Z$ as its value group, which is an
ordered abelian group of rank~1 (that is, a subgroup of the
additive group of $\R$ with the canonical ordering). It is
well-known that if a valuation $v$ for $K$ is of rank~1 then
$K$ is dense in its henselization; see the discussion
in~\cite[p.~53]{engler:prestel:2005}.

One may use a more general method to deal with this problem.
Using the Omitting Types Theorem, another valued field $(L,
w)$ may be constructed such that $(L, w)$ is elementarily
equivalent to $(K, v)$ with respect to the language in
question and $w$ is of rank~1. For example, this method is
used in~\cite{Dickmann1987} to obtain a converse QE result for
real closed valuation rings. We will also use it below to
establish a few converse QE results.

\end{itemize}

Note that Step~2 can always be implemented for any valued field that
is not henselian. So the bulk of the work in the sequel will
concentrate on Step~1 and Step~3.

Next we will describe languages of a quite different kind, namely
the Denef-Pas style languages.

\begin{defn}
Let $K$ be a valued field and $\K$ its residue field. An
\emph{angular component map} is a function $\ac: K \fun \K$ such
that
\begin{enumerate}
 \item $\ac 0 = 0$,
 \item the restriction $\ac \upharpoonright K^{\times}$ is a group
 homomorphism $K^{\times} \fun \K^{\times}$,
 \item the restriction $\ac \upharpoonright (\OO \mi \MM)$ is the
 projection map, that is, $\ac u = u + \MM$ for all $u \in \OO \mi
 \MM$.
\end{enumerate}
\end{defn}

The template of Denef-Pas style languages has three sorts: the field
sort which is the main sort, the residue field sort, and the value
group sort. These are usually denoted by $K$, $\K$, and $\Gamma$.
The $K$-sort and $\K$-sort use the language $\mathcal{L}_R$ of
rings. The $\Gamma$-sort uses the langauge $\mathcal{L}_{OG}$ of
ordered groups, $\{+, <, 0\}$, and an additional symbol $\infty$
that designates the top element in the ordering. There are two
cross-sort function symbols: $v: K \fun \Gamma$, which stands for
the valuation, and $\ac: K \fun \K$, which stands for an angular
component map.

Any language that expands this template is a Denef-Pas language. A
prototypical example is the language $\mathcal{L}_{RRP}$ used
in~\cite{Pa89}, in which the field sort and the residue field sort
use the language $\mathcal{L}_R$ and the $\Gamma$-sort uses the
language $\mathcal{L}_{Pr\infty} = \mathcal{L}_{Pr} \cup
\{\infty\}$, where $\mathcal{L}_{Pr}$ is the Presburger language
$\{+, -, <, 0, 1\} \cup \{D_n: n > 1\}$. Let $S = \langle K, \K,
\Gamma \cup \set{\infty}, v, \ac \rangle$ be a structure of
$\mathcal{L}_{RRP}$. One of the main results of~\cite{Pa89} is
that if $K$ is henselian and both $K$ and $\K$ are of
characteristic 0 then $\Th(S)$ admits QE in the $K$-sort; that is,
for every formula $\varphi$ in $\mathcal{L}_{RRP}$ there is a
formula $\varphi^*$ in $\mathcal{L}_{RRP}$ that does not contain
$K$-quantifiers such that $S \models \varphi \liff \varphi^*$. A
converse of this with respect to henselianity will be established
in Section~\ref{section:denef:pas}.

The following notions are formulated for any Denef-Pas language
$\LL$, where we use $\LL_{K}$, $\LL_{\K}$, and $\LL_{\Gamma\infty}$
to denote the languages used by the three sorts.

\begin{defn}
A formula $\varphi$ in $\LL$ is \emph{simple} if $\varphi$ does not
contain any $K$-quantifiers.
\end{defn}

\begin{defn}
A formula $\varphi$ in $\mathcal{L}_{K} \cup
\mathcal{L}_{\Gamma\infty}$ is a \emph{$\Gamma$-formula} if it does
not contain $K$-quantifiers and atomic formulas in
$\mathcal{L}_{K}$. Similarly a formula $\varphi$ in $\mathcal{L}_{K}
\cup \mathcal{L}_{\K}$ is a \emph{$\K$-formula} if it does not
contain $K$-quantifiers and atomic formulas in $\mathcal{L}_{K}$.
\end{defn}

\section{Henselianity and the Macintyre language}\label{section:macintyre}

In this section we shall describe some conditions under which any
valued field that admits QE in the Macintyre language
$\mathcal{L}_{Mac}$ is henselian. The bulk of the work will
concentrate on the density condition in Step~3. To satisfy that
one can certainly impose some Galois theoretic conditions on $K$
that guarantees that $K$ is dense in its henselization;
see~\cite[Theorem~2.15]{engler:1978:a}. However this does not seem
to be very satisfactory either as it does not bear much on the
intrinsic algebraic structure of the valued field in question.
Below more elementary conditions will be given. An obvious
advantage of this approach is that one can easily construct such
valued fields. We assume that the reader is familiar with the
basics of the theory of valued fields. A good source for this
is~\cite{engler:prestel:2005}.

There will be different conditions depending on whether the
residue characteristic is zero. But first we shall describe some
concepts that are used in these conditions. Let $(L, w)$ be a
valued field. Let $\OO$ be the valuation ring and $\MM$ its
maximal ideal.

For $r,t \in \OO$ we say that they are \emph{comparable}, written
as $r \asymp t$, if there is a natural number $n$ such that either
$w(r^{n}) \leq w(t) \leq w(r^{n+1})$ or $w(t^{n}) \leq w(r) \leq
w(t^{n+1})$. They are \emph{incomparable} if they are not
comparable. We write $r \ll t$ if $r,t$ are incomparable and $w(r)
< w(t)$. If $t \in A \sub L$ and the set $\set{nw(t): n \in
\mathbb{N}}$ is cofinal in the set $\set{w(r) : r \in A}$ then we
say that $t$ is a \emph{cofinal element} in $A$. Note that for all
units $r \in \OO \mi \MM$ and all $s \in \MM$ we have $r \ll t$.
Obviously $r \ll 0$ for any nonzero $r \in \OO$. For $t \in \MM$
we write $\cha (\LOL) \ll t$ if either $\cha (\LOL) = 0$ or $\cha
(\LOL) = p > 0$ and $p \ll t$. If $r \ll t$ for every $r \in A
\sub \OO$ then we simply write $A \ll t$. Similarly we write $A
\asymp t$ if there is an $r \in A$ such that $r \asymp t$ and $r$
is a cofinal element in $A$.

If $R$ is a subring of a field $L$ then we write $R^L$ for the
integral closure of $R$ in $L$. For any $A \sub L$ we write
$\Q(A)$ for the smallest subfield generated by $A$ in $L$. Note
that $\Q(A)^L$ is the algebraic closure of $\Q(A)$ in $L$ and
$(\Q(A) \cap \OO)^L \sub \Q(A)^L \cap \OO$.

\begin{defn}\label{def:prohenselianity}
We say that $(L,w)$ is of \emph{prohenselian degree n} if for any
natural number $1 \leq m \leq n$ the valuation ring $\Q(t_1,
\ldots, t_m)^L \cap \OO$ admits a henselian coarsening for every
sequence $t_1, \ldots, t_m \in \MM$ with $t_m \gg \ldots \gg t_1$.
If $(L,w)$ is of prohenselian degree n for every natural number
$n$ then it is \emph{prohenselian}.
\end{defn}

When does a valuation admit a henselian coarsening? One answer,
Corollary~\ref{henselian:coarsening}, is this: If it lives near a
henselian valuation and is not \emph{antihenselian}:

\begin{prop}\label{antihenselian:conditions}
Let $L^h$ be the henselization of $(L, w)$. The following are
equivalent:
\begin{enumerate}
 \item $L^h$ is the separable closure of $L$.
 \item If $L^*$ is a finite separable extension of $L$ then $w$ has
 $[L^*:L]$ distinct prolongations in $L^*$.
 \item The valuation $w$ is saturated (i.e. $wL$ is divisible and $\LOL$ is
 algebraically closed) and defectless.
\end{enumerate}
If any one of the three conditions is satisfied then $w$ is called
an \emph{antihenselian} valuation.
\end{prop}

\begin{proof}
This is well-known; see, for example, the first section
of~\cite{engler:1978}. A proof can be quite easily assembled from
various results in~\cite[Section~5]{engler:prestel:2005}. For
example, if $\LOL$ is algebraically closed then the inertia group
equals to the decomposition group and if $wL$ is divisible then
the ramification group equals to the inertia group, hence the
inertia field and the ramification field all equal to $L^h$. Since
$w$ is defectless, the ramification field is the separable closure
of $L$.
\end{proof}

Let $\OO_1$ and $\OO_2$ be two valuation rings of the field $L$.
we say that $\OO_1$ and $\OO_2$ are \emph{dependent} if the
smallest subring $\OO_1\OO_2$ of $L$ that contains both $\OO_1$
and $\OO_2$ is a proper subring of $L$.

\begin{thm}[F. K. Schmidt]\label{henselian:dependent:every}
Let $\OO_1$ and $\OO_2$ be two henselian valuation rings of the
field $L$. If $L$ is not separably closed, then $\OO_1$ and
$\OO_2$ are dependent.
\end{thm}
\begin{proof}
See~\cite[Theorem 4.4.1]{engler:prestel:2005}.
\end{proof}

\begin{prop}\label{algebraic:prolongations}
Suppose that $L^*/L$ is an algebraic extension of fields, $\OO$ is
a valuation ring of $L$, and $\OO_1, \OO_2$ are two prolongations
of $\OO$ in $L^*$. If $\OO_1 \sub \OO_2$, then $\OO_1 = \OO_2$.
\end{prop}
\begin{proof}
See~\cite[Lemma 3.2.8]{engler:prestel:2005}.
\end{proof}

\begin{thm}\label{composition:henselian}
Let $\OO \sub \OO_1$ be two valuation rings of $L$ with
corresponding maximal ideals $\MM_1 \sub \MM$. Then $\lbar{\OO} =
\OO / \MM_1$ is a valuation ring of the field $\LOL = \OO_1 /
\MM_1$. The composition $(L, \OO)$ is henselian iff both $(L ,
\OO_1)$ and $(\LOL, \lbar{\OO})$ are henselian.
\end{thm}
\begin{proof}
See~\cite[Corollary 4.1.4]{engler:prestel:2005}.
\end{proof}

From these facts we easily deduce:

\begin{cor}\label{henselian:coarsening}
Let $\OO$ be a henselian valuation ring of $L$. Then for every
non-antihenselian valuation ring $\OO_1$ of $L$ there is a
henselian coarsening $\hat{\OO}_1$ of $\OO_1$.
\end{cor}
\begin{proof}
If $\OO_1$ is henselian then we are done. So assume that $\OO_1$
is not henselian. Since $\OO_1$ is not antihenselian, the
henselization $L^h(\OO_1)$ of $L$ with respect to $\OO_1$ is not
separably closed. Let $\OO_1^h$ be a henselian prolongation of
$\OO_1$ in $L^h(\OO_1)$. Note that such a prolongation may not be
unique. Since $\OO$ is henselian, the unique prolongation $\OO'$
of $\OO$ in $L^h(\OO_1)$ is also henselian. So there is a
valuation ring $\OO'_2$ of $L^h(\OO_1)$ that contains both $\OO'$
and $\OO_1^h$. By Theorem~\ref{composition:henselian} $\OO'_2$ is
henselian. Let $\OO_2 = \OO'_2 \cap L$. By
Proposition~\ref{algebraic:prolongations} $\OO_2$ is a proper
subring of $L$. Since $\OO \sub \OO_2$,  $\OO_2$ is henselian by
Theorem~\ref{composition:henselian} again and contains $\OO_1$, as
desired.
\end{proof}

\begin{rem}
That a field carries a henselian valuation is not a first-order
property in the language $\mathcal{L}_R$. Consider the example
in~\cite[p. 338]{Prestel:Ziegler;1978}. There an inverse limit $L$
of valued fields is constructed such that
\begin{itemize}
  \item $L$ is neither algebraically closed nor real closed,
  \item $L$ is elementarily equivalent to a henselian valued field
  with respect to $\mathcal{L}_R$,
  \item no valuation of $L$ is henselian.
\end{itemize}
\end{rem}

A subgroup $H$ of an ordered abelian group $G$ is \emph{convex}
if, for every $a \in G$, $0 \leq a \leq b$ for some $b \in H$
implies $a \in H$. Obviously the set of all convex subgroups of
$G$ are linearly ordered by inclusion. The order type of this set
is called the \emph{rank} of $G$, denoted by $\rk G$. If $\rk G$
is finite then we identify it with a natural number. For example,
$\rk G = 0$ if and only if $G = \set{0}$. Groups of rank 1, that
is, groups with only one proper convex subgroup $\set{0}$, are of
particular importance for Step~3 in Section~\ref{section:prelim},
because of the following well-known fact:

\begin{fact}
Let $(L, w)$ be a valued field. If the value group $wL$ is of rank
1 then $L$ is dense in the henselization $L^h$ (with respect to
the valuation topology).
\end{fact}

The following characterization of ordered abelian groups of rank 1
has already been mentioned in passing above:

\begin{prop}
A group $G$ is of rank 1 if and only if it is order-isomorphic to
a non-trivial subgroup of the (canonically) ordered additive
subgroup of the reals.
\end{prop}
\begin{proof}
See~\cite[Proposition 2.1.1]{engler:prestel:2005}.
\end{proof}

For the rest of this section let $(K, v)$ be a valued field and
$\OO_v, \MM_v$ its valuation ring and maximal ideal, respectively.

\subsection{The residue characteristic is zero}

Throughout this subsection we assume that $\cha(\K) = 0$, $\Th(K)$
admits QE in $\LL_{Mac}$, and $(K,v)$ is of prohenselian degree 2.
We shall first consider $(K, v)$ as a structure of $\LL_{Mac}$
where, unlike in \p-fields, each predicate $P_n$ is interpreted
naturally as the subgroup of $n$th powers of $K^{\times}$. We do
not assume that $K$ satisfies these other defining conditions for
a \p-adically closed field because they are immaterial to the
discussion below. We shall prove:

\begin{thm}\label{mac:qe:henselian}
Under these conditions, the valuation $v$ is henselian.
\end{thm}

Step~1 in Section~\ref{section:prelim} can be carried out easily
for $(K, v)$.
\begin{lem}\label{carries:henselian:clopen}
For every $n > 1$ the subgroup $P_n$ of $K^{\times}$ is clopen in
the valuation topology induced by $v$.
\end{lem}
\begin{proof}
For every $t \in \MM_v$, $t$ is clearly a cofinal element in
$\Q(t)^K$. The restriction of $v$ to $\Q(t)^K$ admits a henselian
coarsening, which must be $v$ itself as $\rk v\Q(t)^K = 1$. Since
$v(t) > v(n) = 0$ for every $n > 1$. So by Hensel's Lemma $1 + t$
is an $n$th power in $\Q(t)^K$, hence in $K$. So $P_n$ contains an
open neighborhood of 1 in $K$ and hence is open in the valuation
topology induced by $v$. It is also closed as it is a subgroup of
$K^{\times}$.
\end{proof}

Next, note that the relation $v(X) \leq v(Y)$ is not
quantifier-free definable in $\LL_{Mac}$. See the discussion
in~\cite[p. 82]{MMV83}. However, since the relation is definable
in $\LL_{Mac}$, we shall use it as a shorthand for the
corresponding formula in $\LL_{Mac}$. To carry out Step~3 we shall
apply the Omitting Types Theorem to achieve the density condition.
Our goal is to show that the following 2-type
\begin{equation}\label{the:2:type}
\Phi(X, Y) = \set{0 < v(X^n) < v(Y) \wedge Y \neq 0 : n \geq 1}
\end{equation}
is not isolated modulo $\Th(K)$. To that end, we suppose for
contradiction that there is a formula $\pi(X, Y)$ in $\LL_{Mac}$
such that
\begin{itemize}
 \item $\ex{X, Y} \pi(X, Y) \in \Th(K)$ and
 \item $\pi(X, Y) \proves \Phi(X, Y)$ modulo $\Th(K)$.
\end{itemize}

Let $r, t \in \MM_v$ such that $r \ll t$ and $K \models \pi(r,
t)$. Since $K$ admits QE in $\LL_{Mac}$, without loss of
generality we may assume that $\pi(X, Y)$ is of the form
\begin{multline}\label{mac:qf:normal:form}
\bigwedge_i E_i(X, Y) = 0 \wedge F(X, Y) \neq 0
\wedge \bigwedge_k \OO(R_k(X, Y))\\
\wedge \bigwedge_m P_{u_m}(T_m(X, Y)) \wedge \bigwedge_n \neg
P_{u_n}(U_n(X, Y)),
\end{multline}
where $E_i, F, R_k, T_m, U_n\in \Z[X, Y]$. Note that $\pi(X, Y)$
does not contain literals of the form $\neg \OO(S(X, Y))$ with $S
\in \Z[X, Y]$.

The following lemma shows that in fact $\pi(X, Y)$ does not
contain equations.

\begin{lem}\label{no:equation}
For any nonzero polynomial $F(X, Y) \in \Z[X, Y]$, $F(r, t) \neq
0$.
\end{lem}
\begin{proof}
Suppose for contradiction $F(r, t) = 0$. Write $F(X, Y)$ as
\begin{equation}\label{poly:normal:expansion}
F_n(X)Y^n + \ldots + F_0(X),
\end{equation}
where $F_0(X), \ldots, F_n(X) \in \Z[X]$ are not all zero. If
$F(X, Y)$ is a monomial in $Y$ then it can be written as
\begin{equation}\label{poly:normal:subexpansion}
(e_mX^m + \ldots + e_0)Y^i
\end{equation}
for some $0 \leq i \leq n$, where $e_0, \ldots, e_m \in \Z[X]$ are
not all zero. But no two summands in $e_mr^m + \ldots + e_0$ have
the same valuation, for otherwise we would have $v(r) = 0$. Hence
$v(e_mr^m + \ldots + e_0) < \infty$, contradiction.

So we may assume that $F(X, Y)$ has at least two nonzero monomial
summands. Now for some $i > j \geq 0$ we have $v(F_i(r)t^i) =
v(F_j(r)t^j)$. So
\[
v(t^{i - j}) = v(F_j(r)/F_i(r)).
\]
But again, in each $F_k(r)$, no two summands have the same
valuation, so $F_k(r) \ll t$. So $t \asymp r$ at the largest,
contradiction again.
\end{proof}

Now, the formula $\pi(X, Y)$ can actually be satisfied by elements
in $K$ that are comparable.

\begin{lem}\label{powers:moved}
Suppose that $K \models P_u(E(r, t))$, where $E(X, Y) \in \Z[X,
Y]$ are nonzero. Then for sufficiently large natural number $k$
\[
K \models P_u(E(rt^u, t^{ku+1})).
\]
\end{lem}
\begin{proof}
Fix a natural number $k$. Write $E(X, Y)$ as
\[
E_0(X)Y^e\left(\frac{E_n(X)}{E_0(X)}Y^n + \ldots +
\frac{E_1(X)}{E_0(X)}Y + 1 \right)
\]
with $E_0(X), E_n(X) \in \Z[X]$ nonzero. Write $E_0(X)$ as
\[
a_0X^d\left(\frac{a_m}{a_0}X^m + \ldots + \frac{a_1}{a_0}X + 1
\right)
\]
with $a_0, a_m \in \Z$ nonzero.

Let $\hat{v}$ be a henselian coarsening of the restriction of $v$
to $\Q(r,t)^K$. Since clearly $\hat{v}(t) > 0$, we see that
actually
\begin{equation}\label{summation:power}
\Q(r,t)^K \models P_u\left(\frac{E_n(r)}{E_0(r)}t^n + \ldots +
\frac{E_1(r)}{E_0(r)}t + 1 \right).
\end{equation}
Similarly we get
\[
\Q(r)^K \models P_u\left( \frac{a_m}{a_0}r^m + \ldots +
\frac{a_1}{a_0}r + 1\right).
\]
So we must have
\begin{equation}\label{equation:powers:remainder}
K \models P_u(a_0r^dt^e).
\end{equation}
Substituting $rt^u, t^{ku+1}$ for $r,t$ respectively we see that
\[
K \models P_u(a_0r^dt^{ud}t^{kue}t^e).
\]
Applying Hensel's Lemma in $\Q(r,t)^K$ when $k$ is sufficiently
large we deduce that $K \models P_u(E(rt^u, t^{ku+1}))$.
\end{proof}

\begin{lem}\label{mac:type:not:isolated}
Let $u = \prod_m u_m \prod_n u_n$. For sufficiently large natural
number $k$, $K \models \pi(rt^u, t^{ku+1})$. Hence $\pi(X, Y)$
cannot isolate the type $\Phi(X, Y)$ modulo $\Th(K)$.
\end{lem}
\begin{proof}
We have seen that $\pi(X, Y)$ does not contain equations. Also, if
$k$ is sufficiently large then clearly the inequality in $\pi(X,
Y)$ is satisfied by $rt^u, t^{ku+1}$. Hence it remains to show
that for infinitely many $k$
\[
K \models \bigwedge_m P_{u_m}(T_m(rt^u, t^{ku+1})) \wedge
\bigwedge_n \neg P_{u_n}(U_n(rt^u, t^{ku+1})).
\]
Now with the current choice of $u$ and a sufficiently large $k$
clearly the argument for the last lemma works for each $u_m$. On
the other hand, if we run that argument for $\neg P_{u_n}(U_n(r,
t))$ then~(\ref{equation:powers:remainder}) turns into
\[
K \models \neg P_{u_n}(a_0r^dt^e).
\]
So it is easy to see that if $k$ is sufficiently large then $K
\models \neg P_{u_n}(U_n(rt^u, t^{ku+1}))$ for each $n$.
\end{proof}

\begin{thm}
There is a valued field $(L, w)$ such that $w$ is of rank 1 and
$(L, w) \equiv (K, v)$ as structures of $\LL_{Mac}$.
\end{thm}
\begin{proof}
Immediate by the Omitting Types Theorem and the last lemma.
\end{proof}

This shows that Step~3 in Section~\ref{section:prelim} can be
carried out for $(K, v)$.

\subsection{The residue characteristic is nonzero}

Throughout this subsection we assume that $\cha(\K) = p > 0$,
$\Th(K)$ admits QE in $\LL_{Mac}$, and $(K,v)$ is of prohenselian
degree 1. We also assume that $(K,v)$ is \emph{tight}; that is,
$v(p)$ is contained in the smallest nonzero convex subgroup of
$vK$. There is still one more condition for $(K,v)$.

\begin{defn}
Let $\CC$ be a subgroup of $K^{\times}$ such that $\Q^{\times}
\sub \CC$. We say that $\CC$ is \emph{conservative} if
\begin{enumerate}
 \item $p$ is a cofinal element in $\Q(r)^K$ for every $r \in \CC$,
 \item $\CC$ is an existentially closed substructure of $K^{\times}$ over
 $\Q^{\times}$ (that is, with parameters in $\Q^{\times}$)
 with respect to the language $\LL_G$ of groups.
\end{enumerate}
\end{defn}

Let $(L, w)$ be a tight valued field with $\cha(\LOL) = p
> 0$ and $A$ the subfield of algebraic numbers of $L$. Clearly
$\rk wA = 1$. If $(L, w)$ is a \p-adic closed field of \p-rank 1
(or of any \p-rank), then $(A, w)$ is a \p-adic closed field of
\p-rank 1 and, by Macintyre's Theorem, $(A, w)$ is an elementary
substructure of $(L, w)$ with respect to $\LL_{Mac}$. So
$A^{\times}$ is a conservative subgroup of $L^{\times}$. Another
obvious example is when $A$ is a pseudo algebraically closed field
(PAC field), since a field is PAC if and only if it is
existentially closed in every regular extension (with respect to
$\LL_R$, of course). Such valued fields are abundant since every
algebraic extension of a PAC field is PAC. For these and other
basic facts about PAC fields see~\cite[Chapter
11]{fried:jarden:2005}.

Fix a natural number $n$. Suppose that $A^{\times}$ is a
conservative subgroup of $L^{\times}$ and $w$ is a henselian
valuation with $\rk wL > 1$. Now it is actually easy to construct
a valued field $(K, v)$ such that
\begin{itemize}
 \item $\cha(\K) = p > 0$ and $(K, v)$ is tight,
 \item $(K, v)$ is of prohenselian degree n,
 \item there is a conservative subgroup $\CC$ of $K^{\times}$.
\end{itemize}

We start with a subgroup $\CC$ of $A^{\times}$ such that
$\Q^{\times} \sub \CC$ and $\CC$ is an existentially closed
$\LL_G$-substructure of $A^{\times}$ (hence of $L^{\times}$) over
$\Q^{\times}$. Pick an element $t \in L$ with $t \gg p$ and let
$K_0$ be a subfield of $L$ such that $\CC \cup \set{t} \sub K_0$.
Of course the induced valued field $(K_0, w)$ may fail to be of
prehenselian degree $n$. However, since prohenselianity is a sort
of ``closure'' condition for partial henselianity and $(L, w)$ is
henselian, we can simply find a subfield $K_1$ of $L$ such that
\begin{itemize}
 \item $K_0 \sub K_1$,
 \item $\Q(t_1, \ldots, t_n)^L \sub K_1$ for any $t_1, \ldots, t_n \in \MM
 \cap K_0$ with $t_n \gg \ldots \gg t_1$.
\end{itemize}
Then we proceed to find a subfield $K_2$ of $L$ that satisfies the
above two conditions with respect to $K_1$. In this fashion we can
construct a sequence of subfields $K_0, \ldots, K_i, \ldots$ of
$L$ such that $K = \bigcup_i K_i$ is as desired, where the
conservative subgroup in question is $\CC$.

For the rest of this subsection we assume that there is a
conservative subgroup $\CC$ of $K^{\times}$.

\begin{lem}\label{carries:henselian:clopen:mixed:char}
For every $n > 1$ the subgroup $P_n$ of $K^{\times}$ is clopen in
the valuation topology induced by $v$.
\end{lem}
\begin{proof}
If there is a $t \in \MM_v$ with $p \ll t$ then we may simply
repeat the argument in Lemma~\ref{carries:henselian:clopen}. If
$\cha(\K) = p > 0$ is a cofinal element in $K$, then for any $n
> 1$ we consider any $t \in \MM_v$ with $v(t) > 2v(n)$. Since $\rk vK = 1$
and $(K, v)$ is of prohenselian degree 1, the restriction of $v$
to $\Q(t)^K$ is henselian. So by Newton's Lemma $1 + t$ is an
$n$th power in $\Q(t)^K$, hence in $K$.
\end{proof}

\begin{lem}\label{t:going:down}
Let $E_i, F_j \in \Z[X]$ and $x \in \MM_v$ with $x \gg p$ such
that
\[
K \models \bigwedge_i P_{u_i}(E_i(x)) \wedge \bigwedge_j \neg
P_{u_j}(F_j(x)).
\]
Let $u = \prod_i u_i \prod_j u_j$. Then for some $x^*$ with $x^*
\asymp p$
\[
K \models \bigwedge_i P_{u_i}(E_i(x^*)) \wedge \bigwedge_j \neg
P_{u_j}(F_j(x^*)).
\]
\end{lem}
\begin{proof}
Let us begin by considering just one polynomial, say, $E_1(X)$.
Write it as
\begin{equation}\label{poly:one:variabel:normal:form}
a_0X^m \biggl(\frac{a_n}{a_0}X^n + \ldots + \frac{a_{1}}{a_0}X + 1
\biggr),
\end{equation}
where $a_n, \ldots, a_0 \in \Z$, $a_n, a_0 \neq 0$, and $n \geq m
\geq 0$. Let $\hat{v}$ be a henselian coarsening of the
restriction of $v$ to $\Q(x)^K$. Since $x$ is clearly a cofinal
element in $\Q(x)^K$, we may assume that $\hat{v}\Q = 0$. By
Hensel's Lemma we see that
\begin{equation*}
\Q(x)^K \models P_{u_1} \biggl(\frac{a_n}{a_0}x^n + \ldots +
\frac{a_{1}}{a_0}x + 1 \biggr).
\end{equation*}
So we have
\begin{equation}\label{one:variable:div}
K \models P_{u_1}(a_0x^m ).
\end{equation}
It is easy to see that the above argument does not depend on the
number of polynomials under consideration. On the other hand, if
we run the argument for $\neg P_{u_j}(F_j(X))$
then~(\ref{one:variable:div}) turns into
\[
K \models \neg P_{u_j}(a_0x^m ),
\]
for each $j$. So we have
\[
K \models \bigwedge_i P_{u_i}(a_ix^{m_i}) \wedge \bigwedge_j \neg
P_{u_j}(a_jx^{m_j})
\]
for some $a_i, a_j, m_i, m_j \in \Z$. Since $\CC$ is an elementary
$\LL_G$-substructure of $K^{\times}$, there is an $x_* \in \CC$
such that
\[
K \models \bigwedge_i P_{u_i}(a_ix_*^{m_i}) \wedge \bigwedge_j
\neg P_{u_j}(a_jx_*^{m_j}).
\]
Since $p$ is a cofinal element in $\Q(x_*)^K$, we have $\rk
\Q(x_*)^K = 1$ and the restriction of $v$ to $\Q(x_*)^K$ is
henselian. So by Newton's Lemma, for sufficiently large natural
number $k$,
\[
\Q(x_*)^K \models P_{u_i} \biggl(\frac{a_n}{a_0}(p^{ku}x_*)^n +
\ldots + \frac{a_{1}}{a_0}p^{ku}x_* + 1 \biggr).
\]
for each $u_i$, and similarly for each $u_j$. So $x^* = p^{ku}x_*$
for sufficiently large $k$ is as desired.
\end{proof}

Again we use the Omitting Types Theorem to show that Step~3 can be
carried out.

\begin{thm}\label{mac:omitting:type}
There is a valued field $(L, w)$ such that $w$ is of rank 1 and
$(L, w) \equiv (K, v)$ as structures of $\LL_{Mac}$.
\end{thm}
\begin{proof}
It suffices to omit the 2-type~(\ref{the:2:type}). Suppose for
contradiction it is not omitted. Let $\pi(X, Y)$ be as in the last
subsection and $r, t \in \MM_v$ such that $r \ll t$ and $K \models
\pi(r, t)$. Since $(K, v)$ is tight, clearly $p \ll t$. Consider
the existential formula $\ex{X} \pi(X, Y)$. Since $K$ admits QE in
$\LL_{Mac}$, there is a quantifier-free formula $\bigvee_i
\varphi_i(Y)$ in disjunctive normal form such that
\[
K \models \ex{X} \pi(X, Y) \liff \bigvee_i  \varphi_i(Y).
\]
Without loss of generality $K \models \varphi_1(t)$. Then the
proof of Lemma~\ref{no:equation} shows that $\varphi_1(Y)$ does
not contain equations. By Lemma~\ref{t:going:down} there is a
$t^*$ with $t^* \asymp p$ that satisfies all the literals that
involve $n$th power predicates in $\varphi_1(Y)$. It is also clear
from the proof there that $t^*$ may be chosen so that the
inequality in $\varphi_1(Y)$ is also satisfied by $t^*$. This is a
contradiction since $(K, v)$ is tight.
\end{proof}

\subsection{A variation of the Macintyre language}

There is a quite useful variation $\LL_{Mac, D}$ of the Macintyre
language which uses a function instead of a predicate for the
valuation ring. Let $(L, w)$ be a valued field. Define a
restricted division function $D : L^2 \fun L^2$ by
\[
(x, y) \efun \begin{cases}
               x / y, &\text{ if } w(x) \geq w(y) \text{ and } y \neq
               0;\\
               0      &\text{ otherwise.}
       \end{cases}
\]
The behavior of $D$ can be axiomatized; see the definition
in~\cite[p. 82]{MMV83}, where the binary predicate ``$X$~div~$Y$''
can be expressed as a quantifier-free formula $Y = 0 \vee D(Y, X)
\neq 0$. So the language $\LL_{Mac, D}$ is more expressive than
the language $\LL_{Mac}$.

\begin{lem}
Let $\lbar{X}$ be a tuple of variables. Every conjunction
$\varphi(\lbar{X})$ of literals in $\LL_{Mac, D}$ is equivalent to
a disjunction of formulas of the form:
\begin{multline}\label{mac,D:qf:normal:form}
\bigwedge_j \ex{Y_j}(Y_j \neq 0 \wedge Y_j G_j(\lbar{X}) =
H_j(\lbar{X}) \wedge D(Y_j,1) = Y_j)\\
\wedge \bigwedge_i E_i(\lbar{X}) = 0 \wedge F(\lbar{X}) \neq
0\wedge \bigwedge_m P_{u_m}(D(T_m(\lbar{X}),
S_m(\lbar{X})))\\
\wedge \bigwedge_n \neg P_{u_n}(D(U_n(\lbar{X}),V_n(\lbar{X}))),
\end{multline}
where $G_j, H_j, E_i, F, T_m, S_m, U_n, V_n \in \Z[\lbar{X}]$.
\end{lem}
\begin{proof}
Since the function $D$ behaves as division whenever its output is
not 0, the claim essentially says that the ``denominators'' in the
terms can be cleared when the defining conditions for the
occurrences of $D$ are explicitly stated. For example, if $E(X),
F(X) \in \Z[X]$, then $D(E(X),F(X)) = 0$ is equivalent to
\[
\ex{Y}(Y \neq 0 \wedge Y E(X) = F(X) \wedge D(Y, 1) = Y) \vee F(X)
= 0 \vee E(X) = 0.
\]
It is not hard to see that the claim follows from a routine
induction on how deeply the symbol $D$ is nested in $\varphi$.
\qed
\end{proof}

Under the same conditions, the results in the last two subsections
also hold with respect to $\LL_{Mac, D}$.

\begin{thm}
Suppose that $\Th(K)$ admits QE in $\LL_{Mac, D}$ and
\begin{itemize}
 \item if $\cha(\K) = 0$ then $(K,v)$ is of prohenselian degree 2;
 \item if $\cha(\K) = p > 0$ then $(K,v)$ is of prohenselian degree
 1, $(K, v)$ is tight, and there is a conservative subgroup $\CC$
 of $K^{\times}$.
\end{itemize}
Then the valuation $v$ is henselian.
\end{thm}
\begin{proof}
We shall check the three steps in Section~\ref{section:prelim}.

For Step~1 we need to show that, except the equations, all
conjuncts in the form~(\ref{mac,D:qf:normal:form}) define open
sets. By Lemma~\ref{carries:henselian:clopen} and
Lemma~\ref{carries:henselian:clopen:mixed:char} each $n$th power
predicate defines a clopen set. Since quotients of polynomials are
continuous maps, except the equations all the literal conjuncts
in~(\ref{mac,D:qf:normal:form}) define open sets. For the same
reason all the existential conjuncts there define open sets.

As before Step~2 can be carried out in exactly the same way. So we
may find an open set of coefficients that all witness the failure
of henselianity.

For Step~3 we need to show that no formula $\pi(X, Y)$ in
$\LL_{Mac, D}$ of the form~(\ref{mac,D:qf:normal:form}) can
isolate the 2-type~(\ref{the:2:type}). Suppose for contradiction
that there is such a formula $\pi(X, Y)$. Let $r, t \in \MM_v$
such that $r \ll t$ and $K \models \pi(r, t)$.

Suppose that $\cha(\K) = 0$. By Lemma~\ref{no:equation} $\pi(X,
Y)$ does not contain equations. Next, if we run the argument of
Lemma~\ref{powers:moved} for any conjunct
\[
P_{u_m}(D(T_m(X, Y), S_m(X,Y)))
\]
of $\pi(X, Y)$, where $T_m(r,t), S_m(r,t) \neq 0$,
then~(\ref{equation:powers:remainder}) turns into something of the
form
\[
K \models P_{u_m} (D(a_0r^dt^e, b_0r^ft^g)).
\]
So for sufficiently large natural number $k$
\[
K \models P_{u_m}(D(T_m(rt^{u_m}, t^{ku_m + 1}), S_m(rt^{u_m},
t^{ku_m + 1}))).
\]
Similarly we can conclude that, for sufficiently large $k$, the
pair $rt^u, t^{ku + 1}$ satisfies every conjunct in $\pi(X, Y)$
except the existential ones, where $u$ is as in
Lemma~\ref{mac:type:not:isolated}. For the existential conjuncts,
since $\Q \ll r \ll t$, we have
\[
v(ar^dt^e) = v(G_j(r,t)) \leq v(H_j(r,t)) = v(br^ft^g)
\]
for some natural numbers $a, b, d,e,f,g$. So either $e < g$ or $e
= g$ and $d < f$ or $e = g$, $d = f$, and $v(a) \leq v(b)$. So we
see that for sufficiently large $k$
\[
v(G_j(rt^u, t^{ku + 1})) = v(a(rt^u)^d(t^{ku + 1})^e) \leq
v(b(rt^u)^f(t^{ku + 1})^g) = v(H_j(rt^u, t^{ku + 1})).
\]
So indeed we can find a sufficiently large $k$ such that the pair
of comparable elements $rt^u, t^{ku + 1}$ satisfies every conjunct
in $\pi(X, Y)$, which yields a contradiction.

Suppose that $\cha(\K) = p > 0$. So $p \ll t$. There is a formula
$\bigvee_i \varphi_i(Y)$ such that each $\varphi_i(Y)$ is in the
form~(\ref{mac,D:qf:normal:form}) and
\[
K \models \ex{X} \pi(X, Y) \liff \bigvee_i  \varphi_i(Y).
\]
Say, $K \models \varphi_1(t)$. Then the proof of
Lemma~\ref{no:equation} shows that $\varphi_1(Y)$ does not contain
equations. Modifying the proof of Lemma~\ref{t:going:down} as in
the last paragraph we see that there is a $t^*$ with $t^* \asymp
p$ that satisfies all the literals that involve $n$th power
predicates and the inequality in $\varphi_1(Y)$. Moreover for any
$n$ this $t^*$ may be chosen so that $v(t^*) > v(p^n)$. Now, since
$t \gg p$, for each existential conjunct in $\varphi_1(Y)$ we have
\[
v(at^e) = v(G_j(t)) \leq v(H_j(t)) = v(bt^g)
\]
for some natural numbers $a,b, e, g$ with either $e < g$ or $e =
g$ and $v(a) \leq v(b)$. So $t^*$ may be chosen so that
\[
v(G_j(t^*)) = v(a(t^*)^e) \leq v(b(t^*)^g) = v(H_j(t^*)).
\]
So there is a $t^*$ with $t^* \asymp p$ such that $K \models
\ex{X} \pi(X, t^*)$. This is a contradiction since $(K, v)$ is
tight.
\end{proof}

\section{Henselianity without QE}\label{section:henselianity:without:qe}

In the last section we have seen that if a valued field is of
bounded prohenselian degree then QE and other logical conditions
are needed to show henselianity. In this short section we shall
see that prohenselianity may imply henselianity without logical
conditions.

Let $L$ be a field of finite transcendence degree and $w$ a
valuation of $L$. Let $\OO, \MM$ be its valuation ring and maximal
ideal. For any extension of fields $L/K$ we write $\td L/K$ for
the transcendence degree of $L$ over $K$. If $K$ is the prime
field of $L$ then we simply write $\td L$. We shall need the
following fact:

\begin{prop}\label{trdeg:rk:inequ}
Suppose that $L/K$ is a field extension and both $\rk wL$ and $\td
L/K$ are finite. Then
\[
\td \LOL/\K + \rk wL \leq \td L/K + \rk wK.
\]
\end{prop}
\begin{proof}
See~\cite[Corollary 3.4.4]{engler:prestel:2005}.
\end{proof}

By the proof of Lemma~\ref{no:equation} we have $\rk wL \leq \td L
+ 1$. We say that $(L,w)$ is \emph{flat} if
\begin{enumerate}
 \item $\rk wL \geq \td L$;
 \item if $\cha(\LOL) = p > 0$ then $(L,w)$ is tight and if moreover
 $\rk wL = \td L$ then there is a transcendental element $t$ such that
 $\rk w\Q(t)^L = 1$.
\end{enumerate}

Now, if $L$ carries a henselian valuation then it cannot be an
finite extension of $K(t)$, where $K$ is a subfield of $L$ and $t$
is transcendental over $K$;
see~\cite[Proposition~21]{Kuhlmann:F.-V:2006}. Moreover, if every
subfield $K^L \sub L$ carries a henselian valuation and $(K^L, w)$
is not antihenselian then by Corollary~\ref{henselian:coarsening}
$(L, w)$ is prohenselian. It is not hard to block antihenselianity
for each $(K^L, w)$. For example, if $\cha(\LOL) = p > 0$ and
$w(p)$ is not divisible in $wL$ then, for all subfield $K \sub L$,
$w K^L$ is not divisible and hence by
Proposition~\ref{antihenselian:conditions} $(K^L, w)$ is not
antihenselian. Of course if $\rk wL = \td L \leq 1$ then
prohenselianity and henselianity are the same. So the following
proposition is really about valued fields whose value groups have
high ranks.

\begin{prop}
If $(L, w)$ is flat and prohenselian then $w$ is a henselian
valuation.
\end{prop}
\begin{proof}
The proof is by induction on $\td L$. For the base case we have
$\td L = 1$ if $\cha(\LOL) = 0$ or $\td L \leq 1$ if $\cha(\LOL) =
p > 0$ and $\rk wL = 1$. Since in both cases $w$ cannot be
coarsened as $\rk wL =1$, it must be henselian.

Now suppose that $\td L = n + 1$. Since $(L, w)$ is flat, we may
choose a transcendence base $\set{t_1, \ldots, t_{n+1}} \sub \MM$
of $L$ such that
\begin{itemize}
 \item $\cha(\LOL) \ll t_1 \ll \ldots \ll t_{n+1}$ or $\cha(\LOL) \asymp t_1 \ll \ldots \ll t_{n+1}$
 if $\cha(\LOL) = p > 0$ and $\rk wL = n+1$,
 \item $\rk w\Q(t_1)^L = 1$ if $\cha(\LOL) = p > 0$ and $\rk wL = n+1$,
 \item $t_{n+1}$ is a cofinal element in $L$.
\end{itemize}
Let $\hat{w}$ be a henselian coarsening of $w$ and
$\widehat{\LOL}$ the corresponding residue field. In fact we may
assume that $\hat{w}L = wL / \Gamma$, where $\Gamma$ is the
largest proper convex subgroup of $wL$. Note that by the proof of
Lemma~\ref{no:equation} $w\Q(t_1, \ldots, t_n)^L \sub \Gamma$. So
$\rk \hat{w} \Q(t_1, \ldots, t_n)^L = 0$ and the residue field
$\lbar{\Q(t_1, \ldots, t_n)^L}$ with respect to the trivial
valuation is just $\Q(t_1, \ldots, t_n)^L$ itself. Applying
Proposition~\ref{trdeg:rk:inequ} with $K = \Q(t_1, \ldots, t_n)^L$
we get
\[
\td \widehat{\LOL}/\Q(t_1, \ldots, t_n)^L + \rk \hat{w} L \leq \td
L/\Q(t_1, \ldots, t_n)^L = 1.
\]
Hence
\[
\td \widehat{\LOL}/\Q(t_1, \ldots, t_n)^L = 0.
\]
That is, $\widehat{\LOL}$ is algebraic over $\Q(t_1, \ldots,
t_n)^L$.

Since $(\Q(t_1, \ldots, t_n)^L, w)$ is clearly flat and
prohenselian, by the inductive hypothesis the restriction of $w$
to $\Q(t_1, \ldots, t_n)^L$ is henselian. Let $(\widehat{\LOL},
w')$ be the valued field induced by the pair $w, \hat{w}$. There
is an induced valued-field embedding of $(\Q(t_1, \ldots, t_n)^L,
w)$ into $(\widehat{\LOL}, w')$. But $\widehat{\LOL}$ is algebraic
over $\Q(t_1, \ldots, t_n)^L$ and $w$ is a henselian valuation,
clearly $w'$ is also a henselian valuation. Now by
Theorem~\ref{composition:henselian} we conclude that $(L,w)$ is
henselian.
\end{proof}

\section{Henselianity and Denef-Pas style languages}\label{section:denef:pas}

Recall that the three component languages of the prototypical
Denef-Pas language $\mathcal{L}_{RRP}$ are $\LL_R$, $\LL_R$, and
$\LL_{Pr\infty}$. For simplicity, we work with the version of
$\LL_{Pr}$ that does not contain the inverse function symbol $-$.
Throughout this section let $S = \langle K, \K, \Gamma \cup
\set{\infty}, v, \ac \rangle$ be a structure of
$\mathcal{L}_{RRP}$ such that
\begin{enumerate}
 \item $\cha K = 0$,
 \item $\cha \K = 0$,
 \item $v$ and $\ac$ are interpreted as a valuation map and an angular
 component map respectively,
 \item the value group $\Gamma$ is a $\Z$-group,
 \item the theory $\Th(S)$ admits QE in the $K$-sort.
\end{enumerate}

We shall prove:

\begin{thm}\label{converse:qe:rrp}
Under these conditions, the valuation $v$ is henselian.
\end{thm}

The proof of this theorem can be adapted for other Denef-Pas style
languages as well, provided that the value group satisfies certain
mild conditions; see Remark~\ref{other:group:also:do}.

\begin{rem}\label{qe:with:infty}
The theory of $\Z$-groups with a top element in
$\mathcal{L}_{Pr\infty}$ admits QE. This basically follows from
Lemma~5.4 and Lemma~5.5,~\cite{Pa89}.
\end{rem}

In this section the following notational conventions are adopted.
We use $X, Y$, etc. for $K$-sort variables, $M, N$, etc. for
$\Gamma$-sort variables, and $\Xi, \Lambda$, etc. for $\K$-sort
variables. The lowercase of these letters stands for closed terms
or elements in the corresponding sorts. Unless indicated
otherwise, all these letters stand for tuples of variables
whenever they appear in a formula. We use $\lh X$ to denote the
length of $X$. Let $\Z$ and $\lbar{\Z}$ be the rings of integers
of $K$ and $\K$, respectively. Let $\Z_{\Gamma}$ be the smallest
convex subgroup of $\Gamma$.

Every quantifier-free formula in $\mathcal{L}_{RRP}$ is a
disjunction of conjunctions of literals of the following kinds:
\begin{itemize}
 \item Type A: $F(X)\,\, \Box \,\,0$, where $\Box$ is either $=$ or $\neq$ and
 $F(X) \in \Z[X]$.
 \item Type B: $vF_1(X) + M_1 + n_1 \,\, \Box \,\, vF_2(X) + M_2 + n_2$, where
 $\Box$ is one of the symbols $=$, $\neq$, $<$, $>$, and $F_1(X), F_2(X)\in
 \Z[X]$.
 \item Type C: $D_h(vF(X) + M + n)$ or $\neg D_h(vF(X) + M + n)$, where $F(X) \in \Z[X]$.
 \item Type D: $\sum_{i=1}^h G_i(\Lambda)\ac F_i(X) \,\, \Box \,\, 0$,
 where $\Box$ is either $=$ or $\neq$, $F_i(X) \in \Z[X]$, and $G_i(\Lambda)
 \in \lbar{\Z}[\Lambda]$.
\end{itemize}

The following lemma is slightly more general
than~\cite[Lemma~5.3]{Pa89}. Recall the definitions concerning
Denef-Pas style languages in Section~\ref{section:prelim}.

\begin{lem}\label{simple:formula:separated}
Let $\varphi$ be a simple formula in $\mathcal{L}_{RRP}$. Then
$\varphi$ is equivalent to a formula of the form
\[
\bigvee_i (\sigma_i \wedge \chi_i \wedge \theta_i)
\]
where $\sigma_i$ is a quantifier-free formula in $\mathcal{L}_{K}$,
$\chi_i$ a $\K$-formula, and $\theta_i$ a $\Gamma$-formula.
\end{lem}
\begin{proof}
We can write $\varphi$ in its prenex normal form $Q_1 \ldots Q_k
\,\, \psi$ where each $Q_j$ is either a $\Gamma$-quantifier or a
$\K$-quantifier and $\psi$ is a quantifier-free formula. We proceed
by induction on the number $k$ of quantifiers.

If $k = 0$ then $\varphi$ is quantifier-free. So $\varphi$ can be
written in its disjunctive normal form
\[
\bigvee_i (\sigma_i \wedge \chi_i \wedge \theta_i)
\]
where $\sigma_i$ is a conjunction of literals of Type~A, $\chi_i$ a
conjunction of literals of Type~D, and $\theta_i$ a conjunction of
literals of Type~B and Type~C. This proves the base case.

Suppose now $k = l + 1$. So by the inductive hypothesis $\varphi$
can be written in the form
\[
Q_1\,\, \bigvee_i (\sigma'_i \wedge \chi'_i \wedge \theta'_i)
\]
where $\sigma'_i$ is a quantifier-free formula in $\mathcal{L}_{K}$,
$\chi'_i$ a $\K$-formula, and $\theta'_i$ a $\Gamma$-formula. If
$Q_1$ is $\exists \, N$ then we can simply push the quantifier in
and write $\varphi$ as
\[
\bigvee_i (\sigma'_i \wedge \chi'_i \wedge \ex{N} \theta'_i).
\]
If $Q_1$ is $\forall \, N$ then we can rewrite $\bigvee_i
(\sigma'_i \wedge \chi'_i \wedge \theta'_i)$ in its conjunctive
normal form and then push the quantifier in. The other two cases
of $Q_1$ being $\exists \, \Xi$ or $\forall \, \Xi$ are treated in
the same way.
\end{proof}

Simple formulas play an important role in this section. By
Remark~\ref{qe:with:infty} and Lemma~\ref{simple:formula:separated},
they can be written as disjunctions of conjunctions of formulas of
the following forms:

\begin{itemize}
 \item Type I: Same as Type A.
 \item Type II: Same as Type B. Note that, since the conditions $vF(X) = \infty$ and $vF(X) \neq \infty$
 are equivalent to the conditions $F(X) = 0$ and $F(X) \neq 0$ respectively and the latter ones can
 be assimilated into Type I, we may assume that $F(X) \neq 0$ for each $F(X) \in
 \Z[X]$ that appears in a formula of this type.
 \item Type III: Same as Type C. As in Type II we may assume that $F(X) \neq 0$ for each $F(X) \in
 \Z[X]$ that appears in a formula of this type.
 \item Type IV: $\K$-formulas; that is, formulas of the form
  $Q_1 \ldots Q_k \,\, \psi$, where each $Q_j$ is a $\K$-quantifier and
  $\psi$ is a disjunction of conjunctions of literals of Type~D. Again, since
  the conditions $\ac F(X) = 0$ and $\ac F(X) \neq 0$ are equivalent to the
  conditions $F(X) = 0$ and $F(X) \neq 0$, we may assume that $F(X) \neq 0$
  for each $F(X) \in \Z[X]$ that appears in a formula of this type.
\end{itemize}

\subsection{Step~1: Clopen sets}

Since Step~2 and Step~3 do not involve formulas that contain free
$\K$-variables or free $\Gamma$-variables, we may limit our
attention to such formulas of Type I, II, III, and IV. We shall
show that such formulas, except the equalities in the $K$-sort,
define open sets in the corresponding product of the valuation
topology. This takes care of Step~1 in
Section~\ref{section:prelim}.

Since quotients of polynomials are continuous maps with respect to
the valuation topology, that formulas of Type~II define clopen
sets follows from the basic fact that, for $m \in \Gamma$, sets of
the forms $\{x : v(x) = m\}$, $\{x : v(x) > m\}$, etc. are all
clopen in the valuation topology.
See~\cite[Remark~2.3.3]{engler:prestel:2005}.

\begin{lem}\label{type c:clopen}
Let $\varphi(X)$ be a formula of Type~III. Then $\varphi$ defines a
clopen set.
\end{lem}
\begin{proof}
First let $\varphi(X)$ be of the form $D_h(vF(X) + n)$. Let $B \sub
\Gamma$ be the set of all solutions of the formula; that is, $m \in
B$ if and only if $S \models D_h(m  + n)$. For each $m \in \Gamma$
let
\[
A_m = \set{x \in (K^{\times})^e: v F(x) = m},
\]
where $e = \lh X$. Since polynomial maps are continuous, each $A_m$
is clopen in the valuation topology. So
\[
\varphi((K^{\times})^e) = \bigcup_{m \in B} A_m = (K^{\times})^e \mi
\bigcup_{m \notin B} A_m
\]
is clopen. The other case follows immediately from this.
\end{proof}

Let $\OO, \MM$ be the valuation ring and its maximal ideal that
correspond to $v$. The following lemma establishes a crucial
relation between the valuation and the angular component map.

\begin{lem}\label{angular:compo:value:up}
For nonzero $x, y \in K$ with $v(x) = v(y) = m \in \Gamma$, $\ac x =
\ac y$ if and only if $v(x - y) > m$.
\end{lem}
\begin{proof}
If $x = y$ then the lemma is trivial. So we assume further that $x
\neq y$.

For the ``only if'' direction, suppose for contradiction that $\ac
x= \ac y$ but $v(x - y) = m$. So $(x - y)/x$ is a unit. So
\begin{equation*}
\begin{split}
\ac \frac{x-y}{x} &= 1 - \frac{y}{x} + \MM\\
                  &= 1 + \MM - \left(\frac{y}{x} + \MM \right)\\
                  &= 1 + \MM - \ac \frac{y}{x}\\
                  &= 1 + \MM - \frac{\ac y}{\ac x}\\
                  &= 0.
\end{split}
\end{equation*}
So $(x - y)/x = 0$, so $x=y$, contradiction.

For the ``if'' direction, suppose for contradiction that $v(x - y) >
m$ but $\ac x \neq \ac y$. If $m = 0$, that is, $x$ and $y$ are
units in the valuation ring, then
\[
x + \MM = \ac x \neq \ac y = y + \MM.
\]
So $x - y$ is a unit in the valuation ring, that is, $v(x - y) =
0$, contradiction. In general we may consider $1 - y/x$: since
$v(1 - y/x) > 0$ and $y/x$ is a unit, we get $\ac 1 = \ac (y/x)$
by the previous two sentences, so $\ac x = \ac y$.
\end{proof}

\begin{lem}\label{residue:point:open}
Let $\zeta \in \K^{\times}$ and $F(X) \in \Z[X]$. The set
\[
A_{\zeta} = \set{x \in (K^{\times})^e: \ac F(x) = \zeta}
\]
is clopen, where $e = \lh X$.
\end{lem}
\begin{proof}
Let $X = \seq{X_1, \ldots, X_e}$. Write $F(X)$ as $\sum_{i} f_i
G_i(X)$, where $f_i \in \Z$ and each $G_i(X)$ is a unique monomial
in the summation. Let $c$ be a natural number that is larger than
all the exponents of the variables that appear in $F(X)$. Let $x =
\seq{x_1, \ldots, x_e} \in (K^{\times})^e$. For each $n \in \Gamma$
let $\abs{n} = n$ if $n \geq 0$, otherwise $\abs{n} = - n$. For each
$x_j$ with $1 \leq j \leq e$ let
\[
U_{j} = \set{x_j + y : y \in K \text{ and } v(y) > v F(x) +
c\abs{v(x_1)} + \ldots + c\abs{v(x_e)}}.
\]
Note that $x_j \in U_{j}$ and $0 \notin U_{j}$. Clearly each
$U_{j}$ is clopen in the valuation topology. Let
\[
U_x = U_{1} \times \ldots \times U_{e}.
\]
Now each $G_i(X)$ is of the form
\[
X_1^{c_1} \cdots X_e^{c_e}.
\]
For any $\seq{x_1 + y_1, \ldots, x_e + y_e} \in U_x$ we have
\[
f_i (x_1 + y_1)^{c_1} \cdots (x_e + y_e)^{c_e} = f_i x_1^{c_1}
\cdots x_e^{c_e} + H(x, y),
\]
where $y = \seq{y_1, \ldots, y_e}$, $H(X, Y) \in \Z[X, Y]$, and, by
the choice of $U_x$,
\[
v H(x, y) > v F(x).
\]
So
\[
v F(x_1 + y_1, \ldots, x_e + y_e) = v F(x)
\]
and
\[
v (F(x_1 + y_1, \ldots, x_e + y_e) - F(x)) > v F(x).
\]
So by Lemma~\ref{angular:compo:value:up} we get
\[
\ac F(x_1 + y_1, \ldots, x_e + y_e) = \ac F(x).
\]
So
\[
A_{\zeta} = \bigcup_{x \in A_{\zeta}} U_x = (K^{\times})^e \mi
\bigcup_{x \notin A_{\zeta}} U_x
\]
is clopen.
\end{proof}

\begin{rem}\label{why:tau*}
It may seem that we can use the continuity of polynomial maps,
much as in the proof of Lemma~\ref{type c:clopen}, to prove the
above lemma. But this does not work because for $\zeta \in
\K^{\times}$ the set
\[
A_{\zeta} = \set{x \in K: \ac x = \zeta},
\]
although clopen in the valuation topology on $K^{\times}$, is not
closed in the valuation topology on $K$ as there is no open
neighborhood of $0$ that does not intersect with $A_{\zeta}$. This
is the reason why we have chosen to work with the valuation topology
on $K^{\times}$ instead of $K$.
\end{rem}

\begin{lem}\label{clopen:type:IV}
Let $\varphi(X)$ be a formula of Type~IV. Then $\varphi$ defines a
clopen set.
\end{lem}
\begin{proof}
Let $\varphi(X)$ be of the form $Q_1 \ldots Q_k \,\, \psi(X)$
where $\psi(X)$ is a disjunction of conjunctions of formulas of
the form $\sum_{i=1}^h \lambda_i\ac F_i(X) + \xi \,\, \Box \,\, 0$
with $\lambda_i, \xi \in \lbar{\Z}$. Let $B \sub (\K^{\times})^h$
be the set of all solutions of the formula; that is,
$\seq{\zeta_i} \in B$ if and only if
\[
S \models Q_1 \ldots Q_k \,\, \psi^*(\seq{\zeta_i}),
\]
where the formula $\psi^*(\seq{\zeta_i})$ is obtained by replacing
each $F_i(X)$ in $\psi(X)$ with $\zeta_i$. For each $\seq{\zeta_i}
\in (\K^{\times})^h$ let
\[
A_{\seq{\zeta_i}} = \bigcup \bigcap \set{x \in (K^{\times})^e:
\ac F_i(x) = \zeta_i}
\]
be the boolean combination of sets that corresponds to
$\psi^*(\seq{\zeta_i})$, where $e = \lh X$. By
Lemma~\ref{residue:point:open} each $A_{\seq{\zeta_i}}$ is clopen.
So
\[
\varphi((K^{\times})^e) = \bigcup_{\seq{\zeta_i} \in B}
A_{\seq{\zeta_i}} = (K^{\times})^e \mi \bigcup_{\seq{\zeta_i} \notin
B} A_{\seq{\zeta_i}}
\]
is clopen.
\end{proof}

\subsection{Step~3: Omitting a type}

For the rest of this section let $X, Y$ be two single variables.
To carry out Step~3 in Section~\ref{section:prelim} we will again
omit the 2-type~(\ref{the:2:type}) to show:

\begin{thm}\label{2 type:omitted}
There is a structure $S_1 = \langle K_1, \K_1, \Gamma_1 \cup
\set{\infty}, v_1, \ac_1 \rangle$ of $\mathcal{L}_{RRP}$ such that
$S_1 \equiv S$ and $v_1$ is of rank 1.
\end{thm}

\begin{lem}\label{type B:moved}
Let $\varphi(X, Y)$ be a conjunction of formulas of Type~II and~III,
where $X, Y$ are the only free variables. Let $x, y \in \MM$ be
nonzero such that $x \ll y$ and $S \models \varphi(x, y)$. Then for
every natural number $k$ there is an $m \in \Gamma$ with $v(x^k) < m
< v(x^l)$ for some $l
> k$ such that for every $t \in \MM$ with $v(t) = m$ we have
\[
S \models \varphi(x, t).
\]
\end{lem}
\begin{proof}
Let $F_i(X, Y) \in \Z[X, Y]$ run through all the distinct
polynomials that appear in $\varphi(X, Y)$. We may assume that
each $F_i(X, Y)$ is written in the
form~(\ref{poly:normal:expansion})
and~(\ref{poly:normal:subexpansion}). It is not hard to see that
if we choose a $k_0 > 0$ that is larger than the sum of all the
exponents of $X$ that appear in all the $F_i(X, Y)$'s, then, for
each nonzero $t \in \MM$, if $v(t) > v(x^{k_0})$ then
\begin{equation}\label{ploy:reduced}
v (F_i(x, t)) = v(x^{e_i}t^{d_i})
\end{equation}
for some integers $e_i, d_i \geq 0$ with $e_i < k_0$. Clearly in
this situation $e_i, d_i$ are independent of the choice of $t$.
Substituting two free variables $N_1, N_2$ for $v(x), v(t)$
respectively we may rewrite $\varphi(x, t)$ as a formula
$\varphi^*(N_1, N_2)$ in $\mathcal{L}_{Pr\infty}$. So we have
\[
\Gamma \cup \{\infty\} \models \varphi^*(v(x), v(y)).
\]
Now let $v(x) = n$. Let $\Gamma(n)$ be the smallest $\Z$-group
generated by $n$ in $\Gamma$. It is easy to see that the set $\{kn:
k \in \N\}$ is cofinal in $\Gamma(n)$. Clearly $\Gamma(n) \cup
\{\infty\}$ is an elementary substructure of $\Gamma \cup
\{\infty\}$. So for every natural number $k \geq k_0$ we have
\[
\Gamma(n) \cup \{\infty\} \models \ex{N}(kn < N < \infty \wedge
\varphi^*(n, N)).
\]
So for some $m \in \Gamma(n)$ and some $l > k$ we have
\[
\Gamma(n) \cup \{\infty\} \models kn < m < ln \wedge \varphi^*(n,
m).
\]
So for every $t \in \MM$ with $v(t) = m$ we have
\[
\Gamma \cup \{\infty\} \models \varphi^*(n, v(t)).
\]
By the choice of $k_0$ this clearly implies that
\[
S \models \varphi(x, t),
\]
as desired.
\end{proof}

\begin{rem}\label{other:group:also:do}
A close examination of the proof of Lemma~\ref{type c:clopen}
shows that, much as Lemma~\ref{clopen:type:IV}, regardless of what
language the group $\Gamma$ uses and what additional structure it
has, $\Gamma$-formulas without free $\Gamma$-variables always
define clopen sets. Therefore Lemma~\ref{type B:moved} is actually
the only place where we need to use some special properties that
hold in $\Z$-groups, namely
\begin{enumerate}
 \item for any element $n$ in the $\Gamma$-sort the set $\{kn: k \in \N\}$ is cofinal in
 the submodel generated by $n$;
 \item the theory of the $\Gamma$-sort in $\LL_{\Gamma\infty}$ is
 model-complete.
\end{enumerate}
So our converse QE result holds for any group $\Gamma$ and any
language $\LL_{\Gamma}$ such that these two properties are
satisfied.
\end{rem}

\begin{lem}\label{type d:moved}
Let $\varphi(X, Y)$ be a formula of Type~IV, where $X, Y$ are the
only free variables. Let $x, y \in \MM$ be nonzero such that $x
\ll y$ and $S \models \varphi(x, y)$. For every sufficiently large
natural number $k$, if $t \in \MM$ is such that $v(t) \geq v(x^k)$
and $\ac t = \ac y$ then
\[
S \models \varphi(x, t).
\]
\end{lem}
\begin{proof}
Let $F_i(X, Y) \in \Z[X, Y]$ run through all the distinct
polynomials that appear in $\varphi(X, Y)$. As in the previous
lemma we may choose a $k > 0$ that is larger than the sum of all
the exponents of $X$ that appear in all the $F_i(X, Y)$'s so that
for each nonzero $t \in \MM$, if $v(t) > v(x^{k})$ then the
condition~(\ref{ploy:reduced}) holds for each $F_i(X, Y)$. For
such a $t \in \MM$, if $F_i(X, Y)$ is written in the
form~(\ref{poly:normal:expansion})
and~(\ref{poly:normal:subexpansion}), then we have
\[
v(F_b(x)t^b + \ldots + F_0(x)) = vF_0(x)
\]
and
\[
v(F_b(x)t^b + \ldots + F_1(x)t) > vF_0(x)
\]
if $b > 0$, where $F_0(X)$ is written as
\[
X^f(s_aX^a + \ldots + s_0),
\]
with $s_0, \ldots, s_a \in \Z$ and $s_0$ nonzero. So by
Lemma~\ref{angular:compo:value:up} we have
\[
\ac F_i(x, t) = \ac (t^d(F_b(x)t^b + \ldots + F_0(x))) = (\ac t)^d
\cdot \ac F_0(x)
\]
and
\[
\ac F_0(x) = (\ac x)^{f} \cdot \ac s_0.
\]
In particular, since $x \ll y$, we have
\[
\ac F_i(x, y) = (\ac x)^{f} \cdot (\ac y)^d \cdot \ac s_0.
\]
Now if $\ac t = \ac y$ then we have
\begin{equation*}
\ac F_i(x, t) = (\ac x)^{f} \cdot (\ac t)^d \cdot \ac s_0 = (\ac
x)^{f} \cdot (\ac y)^d \cdot \ac s_0 = \ac F_i(x, y).
\end{equation*}
So clearly
\[
S \models \varphi(x, t),
\]
as desired.
\end{proof}

\begin{lem}\label{type:not:isolated}
The 2-type $\Phi(X, Y)$ is not isolated modulo $\Th(S)$.
\end{lem}
\begin{proof}
Suppose for contradiction that there is a formula $\pi(X, Y)$ such
that
\begin{itemize}
 \item $\ex{X, Y} \pi(X, Y) \in \Th(S)$, and
 \item $\pi(X, Y) \proves \Phi(X, Y)$ modulo $\Th(S)$.
\end{itemize}
Since $\Th(S)$ admits QE in the $K$-sort, by
Lemma~\ref{simple:formula:separated}, $\pi(X, Y)$ is equivalent to
a disjunction of conjunctions of formulas of Type I,~II,~III,
and~IV. Without loss of generality we may assume that $\pi(X, Y)$
is just a conjunction of formulas of those four types. Let $x \ll
y$ be such that $S \models \pi(x, y)$. We shall show that there is
a $t \in \MM$ with $x \asymp t$ such that
\[
S \models \pi(x, t).
\]
This yields a contradiction.

By Lemma~\ref{no:equation} $\pi(X, Y)$ cannot contain equalities
in the $K$-sort. Clearly, for sufficiently large $k$, if $t \in
\MM$ is nonzero and $v(t) \geq v(x^k)$ then the pair $(x, t)$
satisfies the inequality in the $K$-sort that appear in $\pi(X,
Y)$. Finally, by Lemma~\ref{type B:moved} and~\ref{type d:moved}
we can choose a sufficiently large $k$ and a $t \in \MM$ with
$v(x^k) < v(t) < v(x^l)$ for some $l > k$ and $\ac t = \ac y$ such
that $S \models \pi(x, t)$, as desired.
\end{proof}

Now Theorem~\ref{2 type:omitted} follows immediately from this lemma
and the Omitting Types Theorem.

\begin{rem}
It is not hard to see that, by considering the formula $\ex{X}
\pi(X, Y)$ as in Lemma~\ref{t:going:down} and
Theorem~\ref{mac:omitting:type}, the proofs in this section can be
modified to cover the case that $\cha(\K) = p > 0$ and $(K, v)$ is
tight. We no longer need the condition that there is a
conservative subgroup since now the theory of the $\Gamma$-sort is
already model-complete.
\end{rem}

\section{Naturality of language}\label{section:general:pers}

In this section we describe a general perspective on QE and
converse QE results. This concerns the usually vague notion that a
language is ``natural'' for a mathematical structure. Here we
propose a precise criterion of naturality by which a language
$\LL$ can be judged with respect to a chosen property $P$:

\begin{crit}\label{criterion:naturality}
Modulo some basic properties (to be specified in context), $\LL$
is natural with respect to $P$ if and only if any structure of
$\LL$ that has $P$ admits QE in $\LL$ and any structure of $\LL$
that admits QE in $\LL$ has $P$.
\end{crit}

In other words, $\LL$ is natural with respect to $P$ if and only
if QE in $\LL$ characterizes $P$. Hence in order to show that
$\LL$ is natural with respect to $P$ one has to show QE and
converse QE.

By Tarski's Theorem and Theorem~\ref{converse:qe:realclosed},
modulo the defining properties of ordered fields, the language
$\LL_{OR}$ is natural with respect to real-closedness. Similarly
by Pas's QE result in~\cite{Pa89} and
Theorem~\ref{converse:qe:rrp}, modulo the other properties
presented at the beginning of Section~\ref{section:denef:pas}, the
language $\LL_{RRP}$ is natural with respect to henselianity.
However, we need the extra conditions described in
Section~\ref{section:macintyre} to establish the naturality of
$\LL_{Mac}$ or $\LL_{Mac, D}$.

Let us examine a simpler case: the $\Z$-groups. A $\Z$-group is a
group that is elementarily equivalent to the group $\Z$ of the
integers in the Presburger language $\mathcal{L}_{Pr}$. By
Presburger's Theorem the theory of $\Z$-groups admits QE in
$\mathcal{L}_{Pr}$. The proof uses the condition that 1 is the
least positive element. However, by
Criterion~\ref{criterion:naturality}, modulo everything else in
the theory, $\mathcal{L}_{Pr}$ is not natural with respect to this
condition. This is a consequence
of~\cite[Corollary~2.11]{weispfenning:1981} which implies that the
structure
\[
\langle \Z \times \Q, +, -, <, 0, 1, D_n \rangle_{n > 1}
\]
admits QE, where everything is interpreted in the standard way
except that $<$ is the lexicographic ordering and the constant~1
designates the element $(1, 0)$.

What about other properties of $\Z$-groups? In every $\Z$-group,
for each divisibility predicate $D_n$, the following holds:
\begin{equation}\label{dn:property}
\fa{x} (D_n(x) \vee D_n(x+1) \vee \ldots  \vee
 D_n(x+n-1)).
\end{equation}
This property is also needed for Presburger's Theorem. Now if we
weaken Criterion~\ref{criterion:naturality} by requiring $P$ to be
equivalent to model-completeness (that is, every formula is
equivalent to an existential formula), then $\mathcal{L}_{Pr}$ is
again not natural with respect to the
property~(\ref{dn:property}). This is again a consequence
of~\cite[Corollary~2.11]{weispfenning:1981} which also implies
that the structure
\[
\langle \Z \times \Z, +, -, <, 0, 1, 1', D_n \rangle_{n > 1}
\]
admits QE, where $<$ is again the lexicographic ordering, the
constant~1 designates the least positive element $(0, 1)$, and the
constant~$1'$ designates the element $(1, 0)$. Note that $(1, 0)$
is not definable with the rest of the structure. But, instead of
$(1, 0)$, any element that satisfies the formula
\[
\fa {x}\bigvee_{i,j < n} D_n(x + i \cdot y + j \cdot 1)
\]
for each $n > 1$ can be used in the QE procedure. So every formula
is equivalent to an existential formula in the reduct of the
structure to $\mathcal{L}_{Pr}$.

\begin{ques}
Is the language $\mathcal{L}_{Pr}$ natural with respect to the
property~(\ref{dn:property}) by
Criterion~\ref{criterion:naturality}? That is, is there a
commutative group with discrete total ordering that admits QE in
$\mathcal{L}_{Pr}$ but does not satisfy the
property~(\ref{dn:property})?
\end{ques}

\bibliographystyle{asl}
\bibliography{MYbib}
\end{document}